\documentclass[a4paper]{amsart}
\title[Higher genus catenoids]{
Higher genus mean curvature $1$ catenoids in hyperbolic and de Sitter 
$3$-spaces}
\date{\today}
\usepackage[dvips]{graphicx} 
\usepackage{amssymb}
\usepackage[usenames]{color}

\usepackage{amsthm}
\theoremstyle{plain}
 \newtheorem{theorem}{Theorem}[section]
 \newtheorem*{theorem*}{Theorem}

 \newtheorem*{lemma*}{Lemma}

 \newtheorem*{fact*}{Fact}

 \newtheorem{lemma}[theorem]{Lemma}
 
\theoremstyle{remark}
 
 \newtheorem{remark}[theorem]{Remark}
 \newtheorem*{remark*}{Remark}
 \newtheorem*{acknowledgements}{Acknowledgements}
 
\numberwithin{equation}{section}
\newcommand{\R}{\boldsymbol{R}}
\newcommand{\C}{\boldsymbol{C}}

\newcommand{\SL}{\operatorname{SL}}

\newcommand{\SU}{\operatorname{SU}}

\newcommand{\id}{\operatorname{id}}

\newcommand{\B}{\mathcal{B}}

\renewcommand{\Re}{\operatorname{Re}}
\renewcommand{\Im}{\operatorname{Im}}

\renewcommand{\phi}{\varphi}
\renewcommand{\epsilon}{\varepsilon}

\author[S.~Fujimori]{Shoichi Fujimori}
\address[Shoichi Fujimori]{%
   Department of Mathematics, Fukuoka University of Education,
   Munakata, Fukuoka 811-4192, Japan}
\email{fujimori@fukuoka-edu.ac.jp}
\author[W.~Rossman]{Wayne Rossman}
\address[Wayne Rossman]{%
   Department of Mathematics, Faculty of Science,
   Kobe University,
   Rokko, Kobe 657-8501, Japan}
\email{wayne@math.kobe-u.ac.jp}

\subjclass[2000]{Primary 53A10; Secondary 65D17.}
\keywords{constant mean curvature 1 surface, higher genus surface, 
          hyperbolic $3$-space, de Sitter $3$-space}
\begin{document}

%%%%%%%%%%%%%%%%%%%%%%%%%%%%%%%%%%%%%%%%%%%%%%%%%%%%%%%%%%%%%%%%%%%%%

\begin{abstract}
We show existence of constant mean curvature 
$1$ surfaces in both hyperbolic $3$-space and de Sitter 
$3$-space with two complete embedded ends and any positive genus 
up to genus twenty.  We also find another such family of surfaces in 
de Sitter $3$-space, but with a different non-embedded end behavior.  
\end{abstract}
\maketitle

%%%%%%%%%%%%%%%%%%%%%%%%%%%%%%%%%%%%%%%%%%%%%%%%%%%%%%%%%%%%%%%%%%%%%
\section*{Introduction}

This paper extends the result in \cite{RS} by K. Sato 
and the second author, and also the 
result in \cite{F2} by the first author.  

In \cite{RS}, it was shown that, although the only complete connected 
finite-total-curvature minimal immersions in
$\R^3$ with two embedded ends are catenoids (Schoen 
\cite{S}), there do exist complete 
connected immersed constant mean curvature (CMC) $1$ surfaces with 
two ends in 
hyperbolic $3$-space $H^3$ that are not surfaces of revolution, 
although such non-rotational surfaces in $H^3$ cannot be embedded 
(Levitt and Rosenberg \cite{LR}).  
The examples found in \cite{RS} are of genus one, but there exist examples 
of genus zero as well, called warped catenoid cousins, which we comment 
on later in this introduction.  This 
comparison is of interest, because minimal surfaces in $\R^3$ and 
CMC $1$ surfaces in $H^3$ are Lawson correspondents, and 
therefore have a very close relationship \cite{B}, \cite{UY1}, 
\cite{UY2}, \cite{RUY1}.  

In \cite{F2}, analogous spacelike surfaces in de Sitter $3$-space $S^3_1$ 
were shown to exist.  Likewise, in this non-Riemannian 
situation, there is a 
similar close relationship between spacelike maximal surfaces in Minkowski 
$3$-space $\R_1^3$ and spacelike CMC $1$ surfaces in $S_1^3$.  
The interest in these surfaces stems largely from 
the nature of their singular sets.  

There is a well-known classical Weierstrass representation for 
minimal surfaces in $\R^3$, and a very similar Weierstrass 
type representation for maximal surfaces in $\R^3_1$ (\cite{K}, 
\cite{UY4} for example).  Because of the relationships described above, 
we have again Weierstrass type representations for CMC $1$ surfaces 
in $H^3$ (\cite{B}, \cite{UY1} for example) and for CMC $1$ 
surfaces in $S_1^3$ (\cite{AA}, \cite{F1}, \cite{FRUYY}).  These 
representations are used here, for $H^3$ and $S_1^3$, 
in Equations \eqref{eq:bryant-dual} and \eqref{eq:bryant-dualB}.  
Furthermore, because of all of these relationships, the 
Osserman inequality for minimal surfaces in $\R^3$ has analogs for
maximal surfaces in $\R^3_1$ (\cite{UY4}), and CMC 1 surfaces in $H^3$ 
(\cite{UY3}) and $S_1^3$ (\cite{F1}, \cite{FRUYY}).  

The examples found in \cite{RS} and \cite{F2} 
were only of genus $1$, and the purpose in this article is to 
show:
\begin{enumerate}
\item the method of \cite{RS} can be extended to give examples 
of any positive genus up to genus twenty, and probably any even higher 
genus as well, without requiring a multi-dimensional 
period problem (showing that a simplification of a 
comment made in the introduction of \cite{RS} is possible), and 
\item in light of recent work on CMC $1$ surfaces with singularities in 
$S^3_1$, the same method will give 
CMC $1$ surfaces of any genus (at least up to genus twenty) 
and two embedded ends in $S^3_1$, and 
\item although the CMC $1$ surfaces in $H^3$ and $S_1^3$ have a similar 
mathematical construction, the behavior 
of the ends of the surfaces in 
$S_1^3$ is more complicated to analyze, related to the fact that 
the group $\SU(1,1)$ used in the $S_1^3$ case is not compact 
(although $\SU(2)$, used in $H^3$, is).  To demonstrate this, 
we find a family of surfaces in $S_1^3$ with hyperbolic ends (the term 
``hyperbolic ends'' was defined in \cite{F1} and \cite{FRUYY}).  
\end{enumerate}

CMC $1$ surfaces with certain kinds of singularities in $S^3_1$ were 
called CMC $1$ faces in \cite{F1} and \cite{FRUYY}.  
Regarding the third point above, in 
\cite{F1} and \cite{FRUYY} it was shown that 
ends of CMC $1$ faces in $S_1^3$ come in three types: elliptic, 
hyperbolic and parabolic.  However, ends of CMC $1$ surfaces in 
$H^3$ will always be elliptic.  Because of this, in the $S_1^3$ 
case an extra argument is needed to demonstrate the numerical 
result just below, and we give that argument at the end of this 
paper.  Our main result is this: 

\medskip

\begin{quote}
{\bf Numerical result:} There exists a one-parameter family of CMC $1$ 
genus $k$ complete properly immersed surfaces in $H^3$ with two 
embedded ends, for any positive integer $k \leq 20$.  Likewise, again 
for any positive integer $k \leq 20$, there exist two one-parameter 
families of genus $k$ CMC $1$ 
faces in $S^3_1$, one with two complete embedded elliptic ends, 
and one with two weakly complete (in the sense of \cite{FRUYY}) 
hyperbolic ends.  
\end{quote}

\medskip

We expect the result is true for many integers 
$k \geq 21$ as well, if not all integers $k \geq 21$.  

Note that the surfaces in $H^3$ and the first family of surfaces 
in $S^3_1$ have the nice property that they are embedded outside 
of a compact set.  

Here we are interested in the case that $k$ is positive, but there do 
exist CMC $1$ surfaces with genus $0$ and embedded ends that are not 
surfaces of revolution.  In the $H^3$ case, they can be found in 
\cite{UY1} (Theorem 6.2) and 
\cite{RUY2} (where they are called warped catenoid 
cousins), and those surfaces in $H^3$ imply the existence of 
corresponding non-rotational examples in $S_1^3$ by Theorem 5.6 
in \cite{F1}.  

The surfaces in the above numerical result are not known to exist 
by any rigorous mathematical method, so, 
like in \cite{RS} and \cite{F2}, we rely on numerics at one step to 
show this result.  In particular, 
we show numerically that a certain continuous function from the real 
line to the real line 
is positive at one point and negative at another, thus implying 
by the intermediate value theorem that it has a zero.  

We provide some graphics of higher genus catenoids in $H^3$, 
see Figure \ref{fg:gkcats}.  (Because the surfaces in $S_1^3$ have 
singularities, making them more difficult to visualize globally, 
the computer graphics become less helpful in this case, and we do 
not show such graphics here.)  

We conclude this introduction with two related remarks:
\begin{enumerate}
\item Although there do not exist any genus $1$ complete connected 
finite-total-curvature minimal immersions in $\R^3$ with two 
embedded ends, there do exist genus $1$ maximal surfaces 
(they can actually be complete 
maxfaces in the sense of \cite{UY4}) with two 
embedded ends in $\R_1^3$.  See Kim-Yang \cite{KY}.  
\item If one allows the ends to be non-embedded, there do exist 
examples of complete connected finite-total-curvature 
minimal surfaces with two ends and positive genus.  See 
Fujimori-Shoda \cite{FS} for example.  
\end{enumerate}

\begin{acknowledgements}
The authors are very grateful to Masaaki Umehara, Kotaro Yamada and 
Seong-Deog Yang for many fruitful discussions, without which 
the authors would not have found the result here.  
\end{acknowledgements}

\begin{figure}[htbp] %%%%%%%%%%%%%%%%%%%%%%%%%%%%%%%%%%%%%%%%%%%%%%%%%%%
\begin{center}
\begin{tabular}{ccc}
 \raisebox{50pt}{$k=1$} & 
 \includegraphics[width=.50\linewidth]{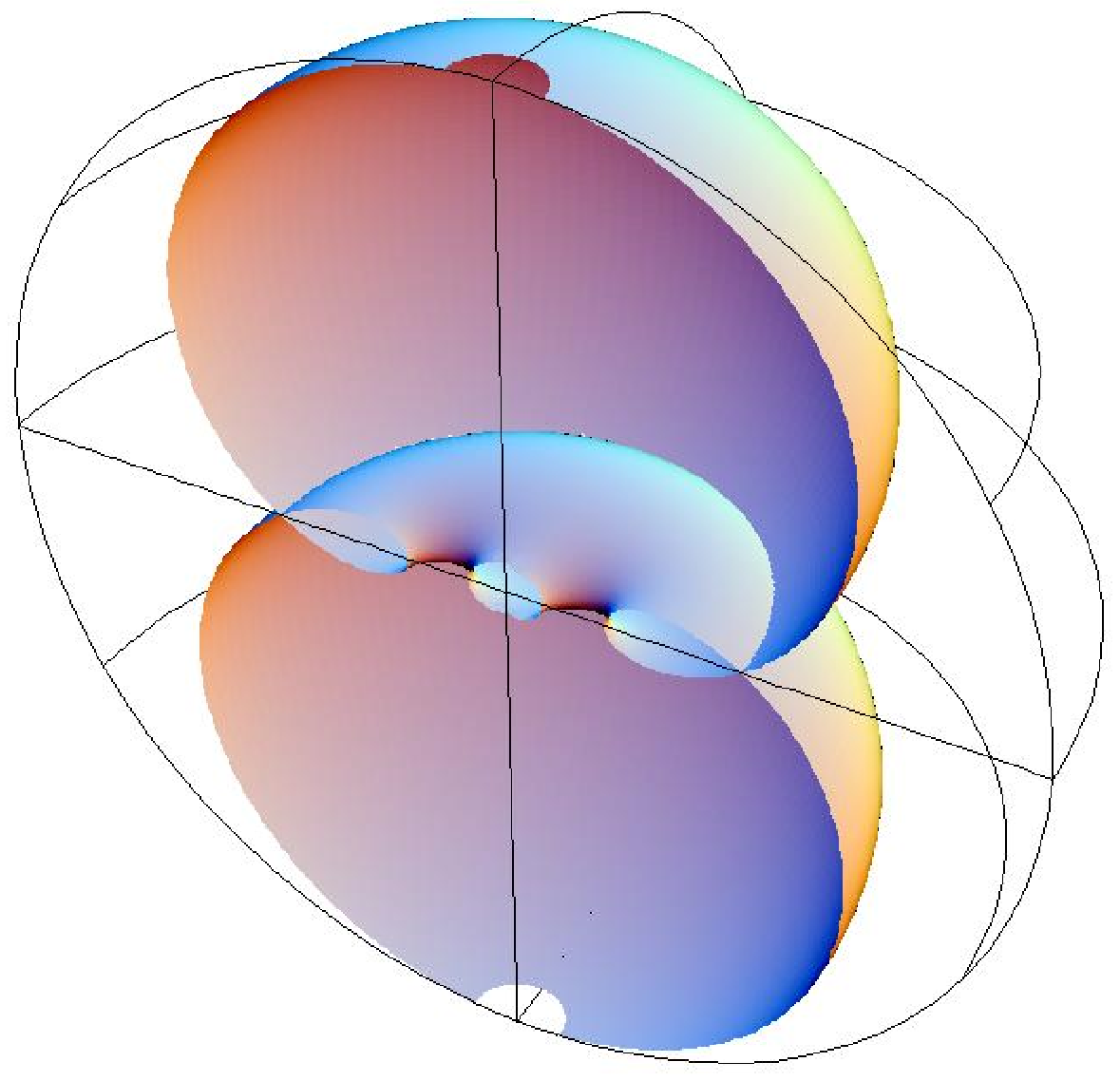} &
 \includegraphics[width=.50\linewidth]{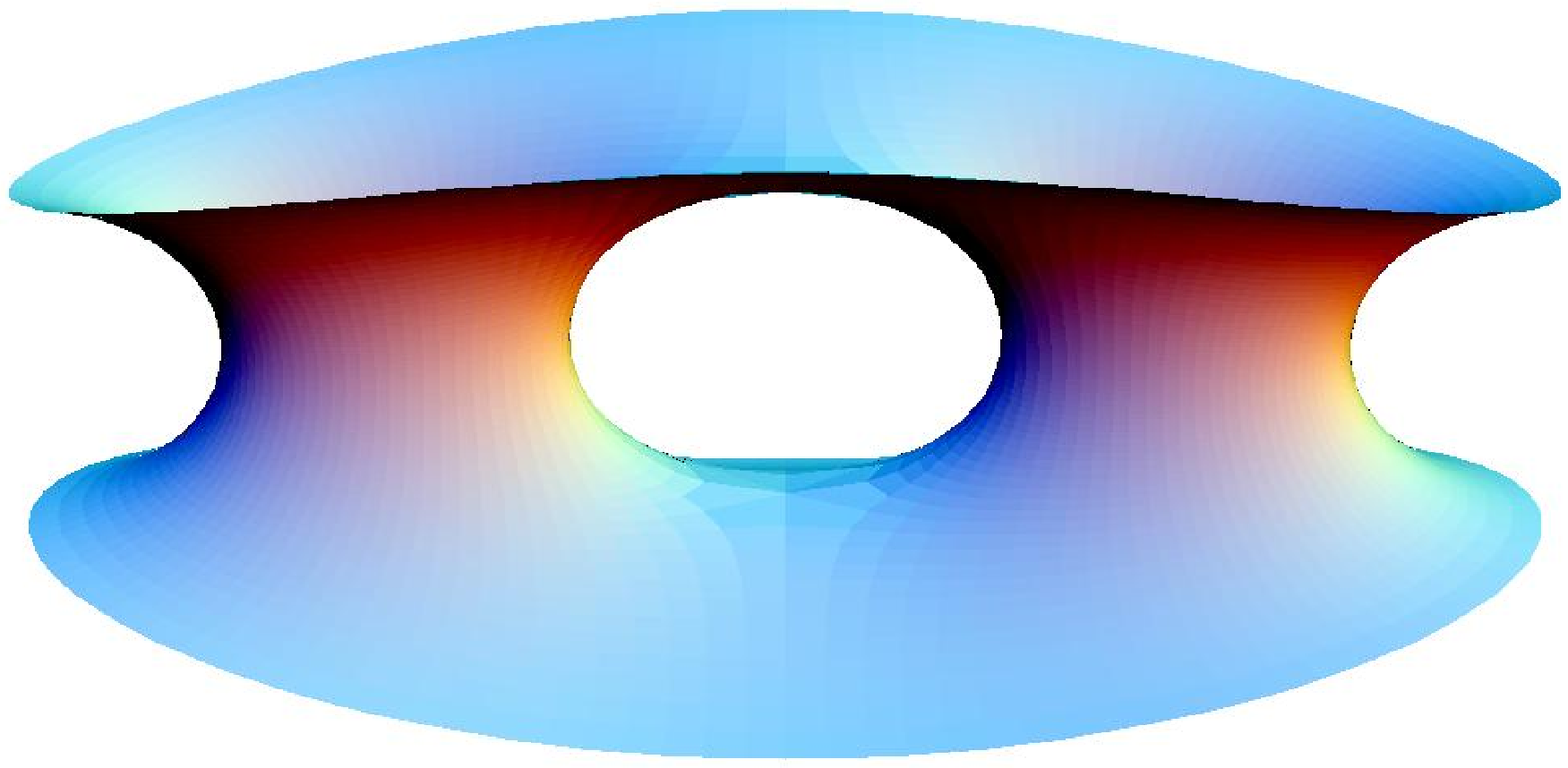} \\
 \raisebox{50pt}{$k=2$} & 
 \includegraphics[width=.50\linewidth]{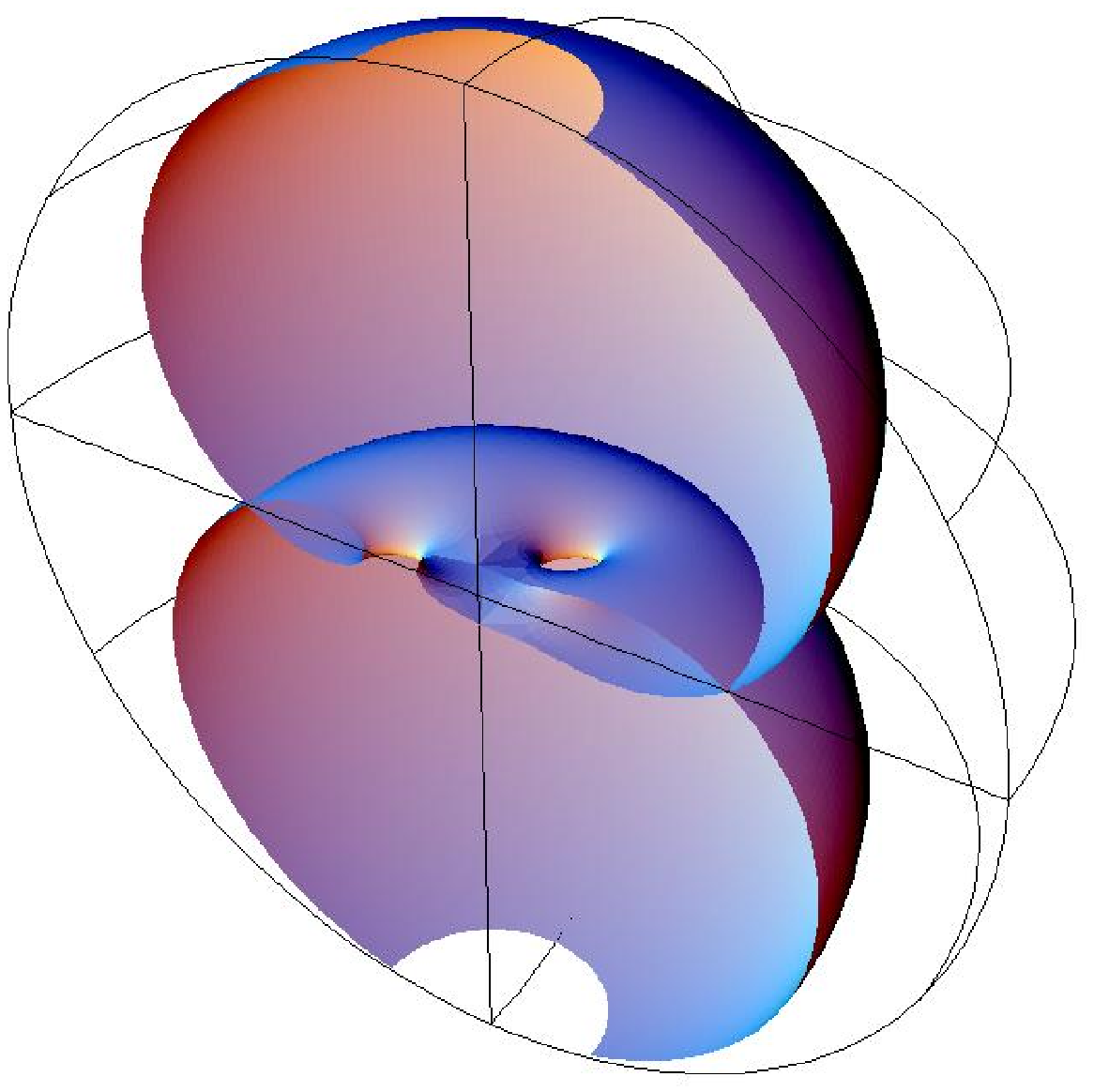} &
 \includegraphics[width=.50\linewidth]{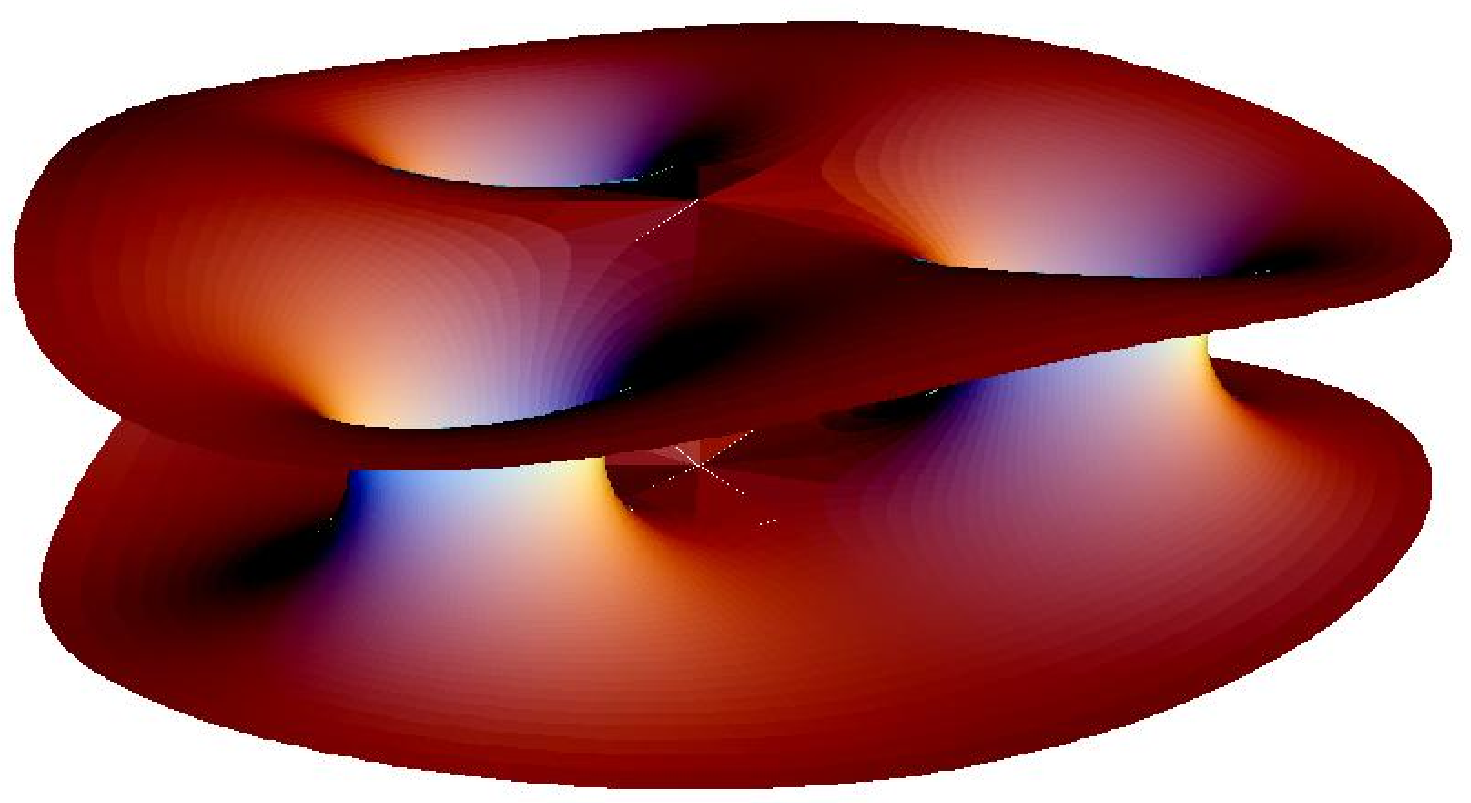} \\
 \raisebox{50pt}{$k=3$} & 
 \includegraphics[width=.50\linewidth]{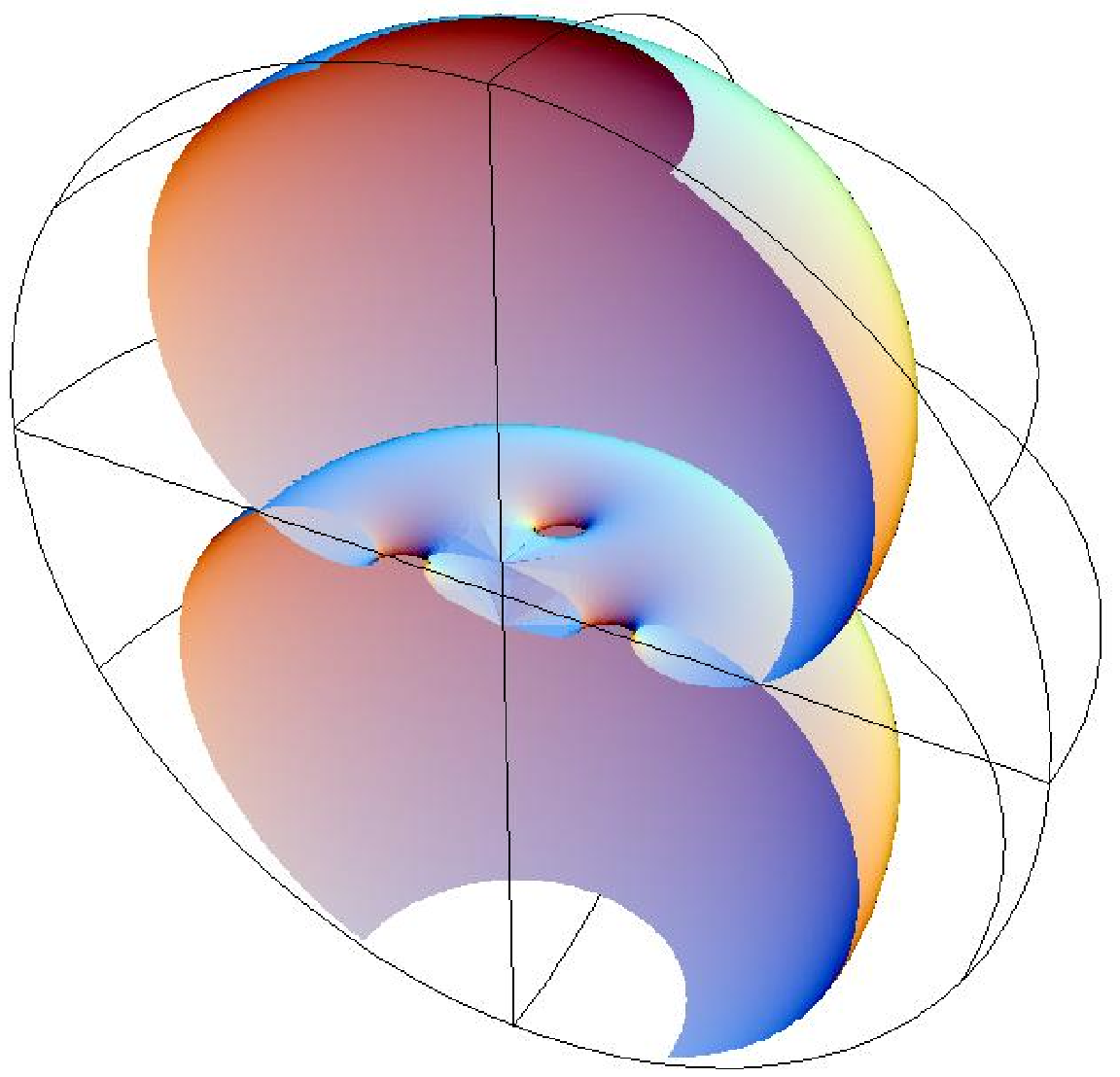} &
 \includegraphics[width=.50\linewidth]{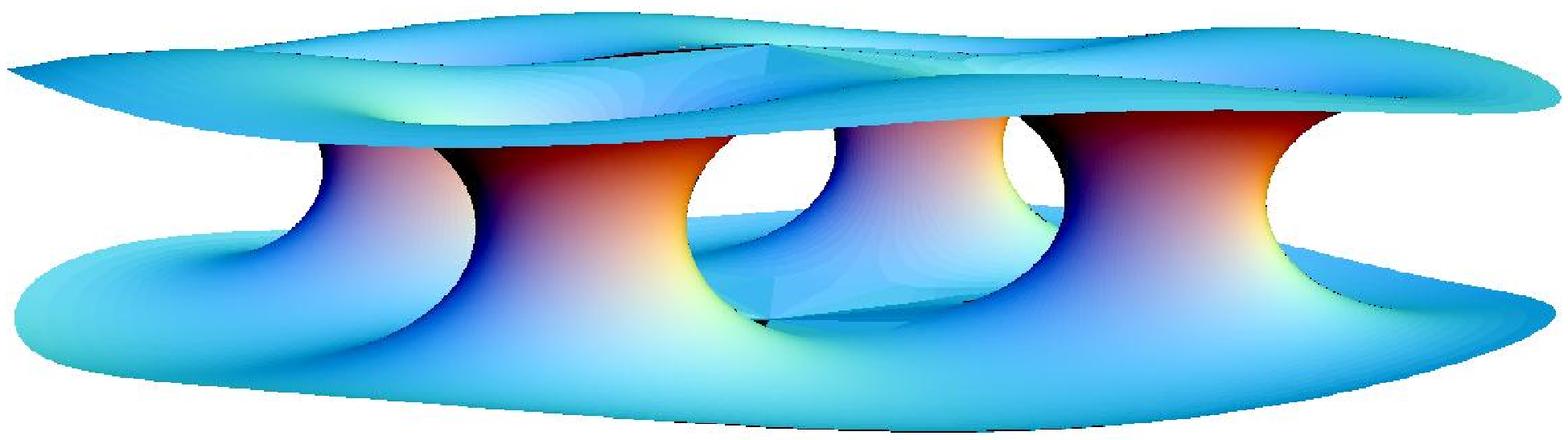} \\
 \raisebox{50pt}{$k=8$} & 
 \includegraphics[width=.40\linewidth]{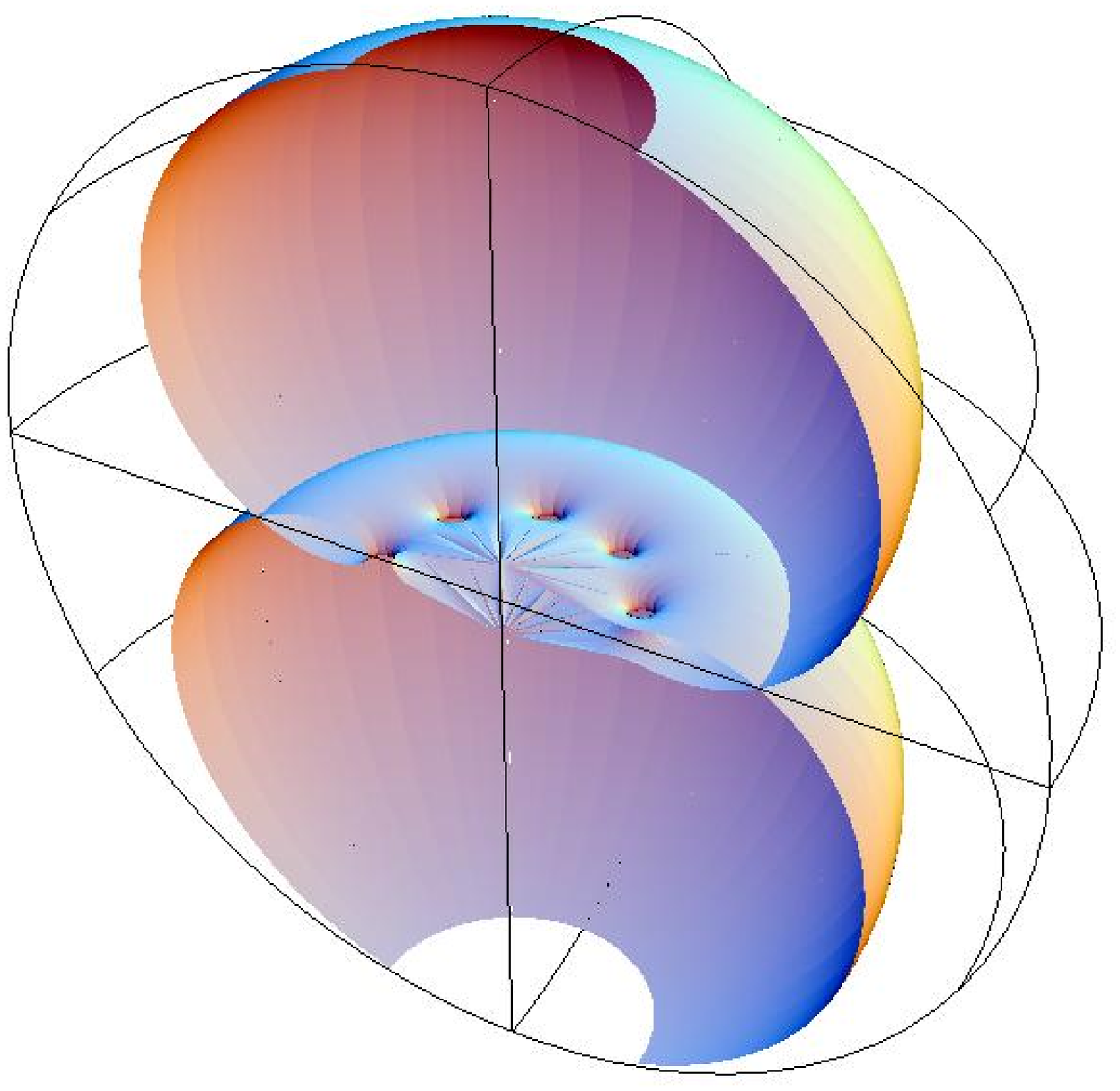} &
 \includegraphics[width=.50\linewidth]{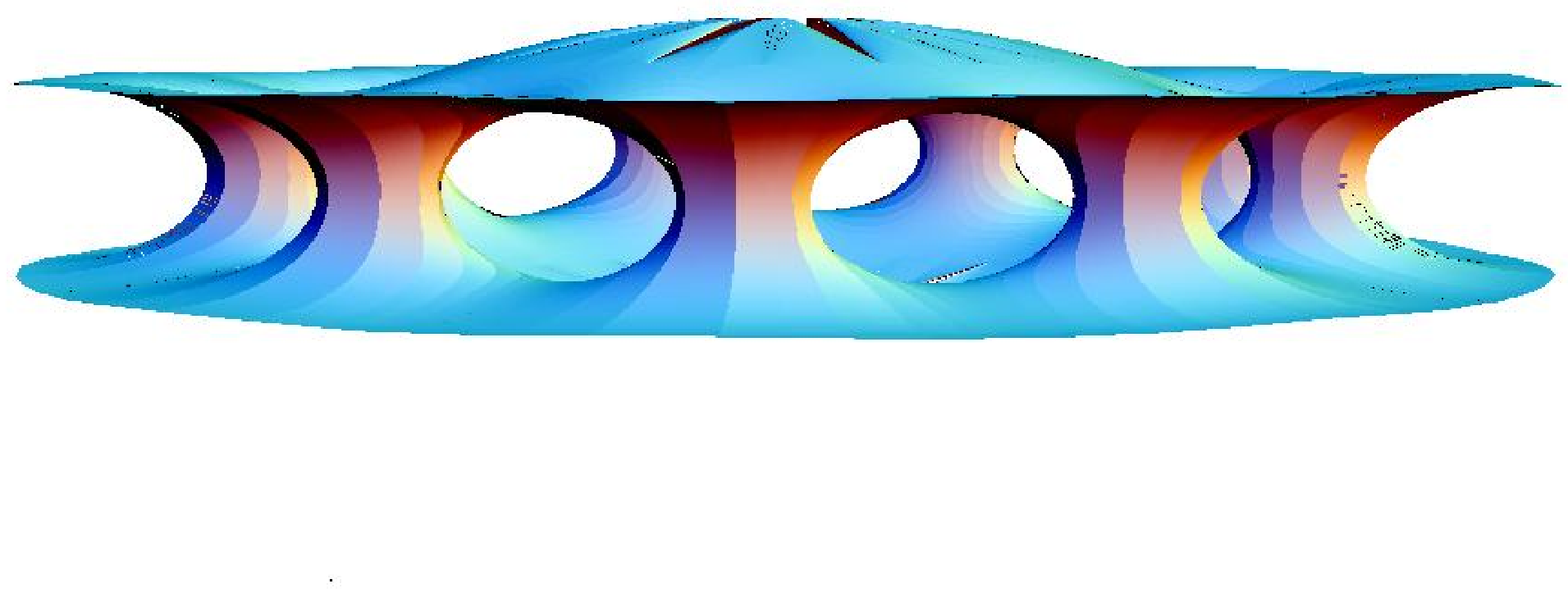} \\
\end{tabular}
\end{center}
\caption{Half cut-away of higher genus catenoids in $H^3$ (in 
         the Poincar\'e ball model), on the left, 
         and central portions of those same surfaces on the 
         right, for $k=1$, $k=2$, $k=3$ and $k=8$.}
\label{fg:gkcats}
\end{figure} %%%%%%%%%%%%%%%%%%%%%%%%%%%%%%%%%%%%%%%%%%%%%%%%%%%%%%%%%%%

\section{The Weierstrass data}

Here we use a more general Weierstrass data than in \cite{RS} and 
\cite{F2}, allowing the genus to be any positive number.  

Take the compact Riemann surface 
\[ \overline{M} = \left\{ 
(z,w) \in (\C \cup \{ \infty \})^2 \, \left| 
\, w^{k+1} = z \left( \frac{z-\lambda^{-1}}{\lambda-z} 
\right)^k \right. \right\} \; , \] 
where $k$ is any positive integer and $\lambda$ is 
any real constant such that $\lambda > 1$.  
This $\overline{M}$ has the structure of a Riemann 
surface, and $z$ provides a local complex coordinate for 
$\overline{M}$ at all but four points.  
At those four points $(0,0)$, 
$(\infty,\infty)$, $(\lambda,\infty)$ and $(\lambda^{-1},0)$, 
we can take a local coordinate $\zeta_0$,
$\zeta_\infty$, $\zeta_\lambda$
and $\zeta_{\lambda^{-1}}$ satisfying $\zeta_0^{k+1} = z$,
$\zeta_\infty^{-k-1} = z$, $\zeta_\lambda^{k+1} = z-\lambda$
and $\zeta_{\lambda^{-1}}^{k+1} = z-\lambda^{-1}$, respectively.
By applying the Riemann-Hurwitz 
relation, we find that this Riemann surface  
has genus $k$.  Then take 
\[ M = \overline{M} \setminus \{ (0,0),(\infty,\infty) \} \; . \] 
Here $(0,0)$ and $(\infty,\infty)$ will represent the two 
ends of the surfaces we will construct (surfaces having 
domain $M$).  Let $\widetilde{M}$ be the universal cover of $M$.  

We now take the Weierstrass data 
\begin{equation}\label{eq:w-data} 
G = \lambda^{k/(k+1)} w \; , \;\;\; 
\Omega = c \cdot \frac{dz}{zw} \; . \end{equation}
Here $c$ is any nonzero real constant.  

   \begin{remark}\label{rem:onthechoiceofG}
    Multiplying the hyperbolic Gauss map $G$ by a constant is equivalent 
    to just a rigid motion of the surface, in both $H^3$ and $S^3_1$.  
    So we could have chosen $G$ to be $G=w$ in Equation \eqref{eq:w-data}, 
    as our goal is to construct surfaces in $H^3$ and $S^3_1$.
    However, when considering relations with minimal surfaces in $\R^3$ and
    maximal surface in $\R^3_1$, the choice of $G$ in \eqref{eq:w-data}
    will prove to be useful, as we will see in 
    Remark~\ref{rm:min-max}.  So here we use
    the hyperbolic Gauss map $G$ as given in \eqref{eq:w-data}. 
   \end{remark}

Note that $\deg(G)=k+1$.  Now take a solution 
\[ F : \widetilde{M} \to \SL(2,\C) \] of 
\begin{equation}\label{eq:bryant-dual} dF \cdot F^{-1} = 
\begin{pmatrix} G & -G^2 \\ 1 & -G \end{pmatrix} \Omega \; . \end{equation}
Noting that, with $a^*=\bar a^t$, 
\[
H^3=\{aa^*\,|\,a\in\SL (2,\C)\},
\quad
S^3_1=\{ae_3a^*\,|\,a\in\SL (2,\C)\},
\quad
e_3=\begin{pmatrix}1&0\\0&-1\end{pmatrix}, 
\]
define 
\begin{equation}\label{eq:bryant-dualB}
f_H:=FF^*:\widetilde{M}\to H^3 
\quad\text{and}\quad
f_S:=Fe_3F^*:\widetilde{M}\to S^3_1.
\end{equation}
Then $f_H$ (resp. $f_S$) gives a CMC $1$ immersion in $H^3$ 
(resp. a CMC $1$ face in $S^3_1$), since 
\begin{equation}\label{liftmetric}
\left( 1+|G|^2 \right)^2 \Omega\overline{\Omega}
\end{equation}
gives a positive definite metric on $M$,  
see \cite{UY3} and Theorem 1.9 of \cite{F1}.  
Note that the Hopf differential of both $f_H$ and $f_S$ is written as 
\[
Q=\Omega dG
 =\frac{c\lambda^{k/(k+1)}}{k+1}
  \cdot\frac{z^2+\{(k-1)\lambda^{-1}-(k+1)\lambda\}z+1}
            {z^2(z-\lambda^{-1})(z-\lambda)}dz^2.
\]

\section{Symmetries of the surface}

We define 
\[
w_0:=\lambda^{-k/(k+1)}\in\R \quad\text{and}\quad 
\Lambda:=e^{2\pi i/(k+1)}.
\]
Then it is easy to see that 
\[
(1,\Lambda^j w_0)\in M
\quad\text{for any $j=0,1,\dots , k$}.
\]

Consider the symmetries 
\[
\kappa_1(z,w)=\left(\bar z,\bar w\right),\quad
\kappa_2(z,w)=\left(\frac{1}{z},\frac{1}{\lambda^{2k/(k+1)}} 
                    \frac{1}{w} \right),\quad
\kappa_3(z,w)=\left(\bar z,\Lambda \bar w\right)
\]
on $\overline{M}$.  
Then we have the following, which follows from a proof analogous 
to proofs found in \cite{RS} and \cite{F2}: 

\begin{lemma}\label{lm:rs-first}
Let 
\[
F(z,w)=\begin{pmatrix}A&B\\C&D\end{pmatrix}
\]
be a solution of \eqref{eq:bryant-dual} with the initial condition 
$F(1,w_0)=\id$.  Then 
\begin{equation}\label{eq:RSlemma51}
F(\kappa_1(z,w))=\begin{pmatrix}\bar A&\bar B\\\bar C&\bar D\end{pmatrix},
\quad
F(\kappa_2(z,w))=\begin{pmatrix}D&C\\B&A\end{pmatrix},
\quad
F(\kappa_3(z,w))
=\begin{pmatrix}\bar A&\Lambda\bar B\\\Lambda^{-1}\bar C&\bar D\end{pmatrix}.
\end{equation}
\end{lemma}

\begin{proof}
Note that we have the following relations under the symmetries $\kappa_j$: 
\begin{eqnarray*}
& G \circ \kappa_1 = \bar G \; , \;\;\; 
G \circ \kappa_2 = G^{-1} \; , \;\;\; 
G \circ \kappa_3 = \Lambda \bar G \; , 
& \\ &
\kappa_1^* \Omega = \bar \Omega \; , \;\;\; 
\kappa_2^* \Omega = -\lambda^{2k/(k+1)} w^2 \Omega \; , \;\;\; 
\kappa_3^* \Omega = \Lambda^{-1} \bar \Omega \; . &
\end{eqnarray*}
It follows that 
\begin{align*}
\kappa_1^* \begin{pmatrix}
G & -G^2 \\ 1 & -G \end{pmatrix}
\Omega &= \overline{\begin{pmatrix}
G & -G^2 \\ 1 & -G \end{pmatrix}
\Omega} \; , \\
\kappa_2^* \begin{pmatrix}
G & -G^2 \\ 1 & -G \end{pmatrix}
\Omega &=  
\begin{pmatrix}0&i\\i&0\end{pmatrix}^{-1}
\begin{pmatrix} G & -G^2 \\ 1 & -G \end{pmatrix} \Omega
\begin{pmatrix}0&i\\i&0\end{pmatrix} \; , \\
\kappa_3^* \begin{pmatrix}
G & -G^2 \\ 1 & -G \end{pmatrix}
\Omega &= 
\begin{pmatrix}
\sqrt{\Lambda} & 0 \\ 0 & 1/\sqrt{\Lambda} 
\end{pmatrix}
\overline{\begin{pmatrix}
G & -G^2 \\ 1 & -G \end{pmatrix}
\Omega} 
\begin{pmatrix}
1/\sqrt{\Lambda} & 0 \\ 0 & \sqrt{\Lambda} 
\end{pmatrix} \; .  
\end{align*}
Because the initial condition 
$F(1,w_0)=\id$ satisfies 
\begin{eqnarray*} 
&
\overline{F(1,w_0)} = F(1,w_0) \; , \qquad
\begin{pmatrix}0&i\\i&0\end{pmatrix}^{-1}\!\!\!
F(1,w_0) \begin{pmatrix}0&i\\i&0\end{pmatrix}
 = F(1,w_0) \; , 
 & \\ &
\begin{pmatrix}
\sqrt{\Lambda} & 0 \\ 0 & 1/\sqrt{\Lambda} 
\end{pmatrix}
\overline{F(1,w_0)} 
\begin{pmatrix}
1/\sqrt{\Lambda} & 0 \\ 0 & \sqrt{\Lambda} 
\end{pmatrix}
= F(1,w_0) \; , 
&
\end{eqnarray*} 
the lemma follows.  
\end{proof}

We now consider two loops in $M$ (see Figure \ref{fg:loops}): 

\begin{figure}[htbp] %%%%%%%%%%%%%%%%%%%%%%%%%%%%%%%%%%%%%%%%%%%%%%%%%%%
\begin{center}
\unitlength=1.0pt
\begin{picture}(300,100)
\put(50,50){\line(1,0){200}}%real axis
\thicklines
%\put(150,0){\line(0,1){100}}%imaginary axis
\put(50,50){\line(1,0){20}}%real axis
\put(110,50){\line(1,0){80}}%real axis
\thinlines
\put(70,50){\circle*{2}}%point at 0
\put(110,50){\circle*{2}}%point at lambda^-1
\put(150,50){\circle*{2}}%point at 1
\put(190,50){\circle*{2}}%point at lambda
\put(70,42){\makebox(0,0)[cc]{$0$}}
\put(114,44){\makebox(0,0)[cc]{$\lambda^{-1}$}}
\put(150,42){\makebox(0,0)[cc]{$1$}}
\put(190,44){\makebox(0,0)[cc]{$\lambda$}}
\put(116,27){\makebox(0,0)[cc]{$\gamma_1$}}
\put(184,35){\vector(-1,0){0}}
\put(116,35){\vector(-1,0){0}}
\bezier57(150,50)(167,35)(184,35)%gamma1
\bezier40(184,35)(199,35)(199,50)%gamma1
\bezier29(155,55)(167,65)(184,65)%gamma1
\bezier20(184,65)(199,65)(199,50)%gamma1
%\bezier12(199,50)(199,65)(184,65)%gamma1
%\bezier36(184,65)(150,65)(116,65)%gamma1
%\bezier12(116,65)(101,65)(101,50)%gamma1
\bezier40(101,50)(101,35)(116,35)%gamma1
\bezier57(116,35)(133,35)(150,50)%gamma1
\bezier20(101,50)(101,65)(116,65)%gamma1
\bezier29(116,65)(133,65)(145,55)%gamma1
\bezier11(155,55)(150,50.5)(145,55)%gamma1
\put(100,98){\makebox(0,0)[cc]{$\gamma_2$}}
\put(100,90){\vector(1,0){0}}
\bezier150(150,50)(125,10)(100,10)%gamma2
\bezier150(100,10)(60,10)(60,50)%gamma2
\bezier150(60,50)(60,90)(100,90)%gamma2
\bezier150(100,90)(125,90)(150,50)%gamma2
\end{picture}
\end{center}
\caption{Projection to the $z$-plane of the loops $\gamma_1$ and $\gamma_2$.
         (Note that $z=1$ is not a branch point of $M$, and $\gamma_1$ starts at 
         one lift of $z=1$, and then passes through another lift of $z=1$, 
         and then returns to the first lift of $z=1$.)}
\label{fg:loops}
\end{figure}
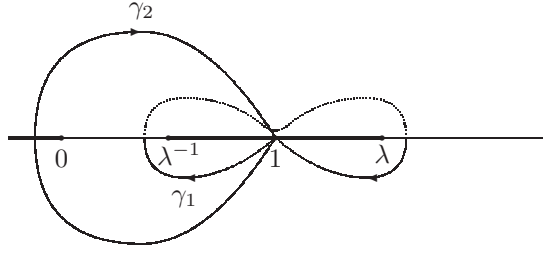 %%%%%%%%%%%%%%%%%%%%%%%%%%%%%%%%%%%%%%%%%%%%%%%%%%%%%%%%%%%

\begin{itemize}
\item 
The loop $\gamma_1:[0,1]\to M$ starts at $\gamma_1(0)=(1,w_0)\in M$. 
Its first portion has $z$ coordinate in $\{\Im(z)<0\}$ and ends at a point 
$(z,w)$ where $z\in\R$ and $0<z<\lambda^{-1}$.  
Its second portion starts at $(z,w)$ and ends at $(1,\Lambda w_0)$ and has $z$ 
coordinate in $\{\Im(z)>0\}$.  
Its third portion starts at $(1,\Lambda w_0)$ and ends at 
$(1/z,\Lambda /(\lambda^{2k/(k+1)}w))$ and has $z$ coordinate in 
$\{\Im(z)>0\}$.  
Its fourth and last portion starts at $(1/z,\Lambda /(\lambda^{2k/(k+1)}w))$ 
and returns to the base point $\gamma_1(1)=(1,w_0)$ and has $z$ coordinate in 
$\{\Im(z)<0\}$.  
\item
The loop $\gamma_2:[0,1]\to M$ starts at $\gamma_2(0)=(1,w_0)$.  
Its first portion has $z$ coordinate in $\{\Im(z)<0\}$ and ends at a 
point $(z,w)$ where $z\in\R$ and $z<0$.  
Its second and last portion starts at $(z,w)$ and returns to 
$\gamma_2(1)=(1,w_0)$ and has $z$ coordinate in $\{\Im(z)>0\}$.  
\end{itemize}

We will also consider the following two paths (not loops) in $M$ 
(see Figure \ref{fg:curves}): 

\begin{figure}[htbp] %%%%%%%%%%%%%%%%%%%%%%%%%%%%%%%%%%%%%%%%%%%%%%%%%%%
\begin{center}
\unitlength=1.0pt
\begin{picture}(300,60)
\put(50,50){\line(1,0){200}}%real axis
\thicklines
%\put(150,0){\line(0,1){100}}%imaginary axis
\put(50,50){\line(1,0){20}}%real axis
\put(110,50){\line(1,0){80}}%real axis
\thinlines
\put(70,50){\circle*{2}}%point at 0
\put(110,50){\circle*{2}}%point at lambda^-1
\put(150,50){\circle*{2}}%point at 1
\put(190,50){\circle*{2}}%point at lambda
\put(70,42){\makebox(0,0)[cc]{$0$}}
\put(114,44){\makebox(0,0)[cc]{$\lambda^{-1}$}}
\put(150,42){\makebox(0,0)[cc]{$1$}}
\put(190,44){\makebox(0,0)[cc]{$\lambda$}}
\put(116,27){\makebox(0,0)[cc]{$c_1$}}
\put(116,35){\vector(-1,0){0}}
\bezier40(101,50)(101,35)(116,35)%c1
\bezier57(116,35)(133,35)(150,50)%c1
\put(100,2){\makebox(0,0)[cc]{$c_2$}}
\put(100,10){\vector(-1,0){0}}
\bezier150(150,50)(125,10)(100,10)%c2
\bezier150(100,10)(60,10)(60,50)%c2
\end{picture}
\end{center}
\caption{Projection to the $z$-plane of the curves $c_1$ and $c_2$.}
\label{fg:curves}
\end{figure}
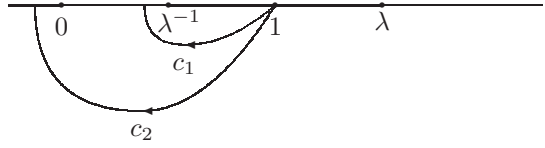 %%%%%%%%%%%%%%%%%%%%%%%%%%%%%%%%%%%%%%%%%%%%%%%%%%%%%%%%%%%
\begin{itemize}
\item 
Let $c_1:[0,1]\to M$ be a curve starting at $c_1(0)=(1,w_0)$ 
whose projection to the $z$-plane is an embedded curve in $\{\Im(z)<0\}$, 
and whose endpoint $c_1(1)$ has a $z$ coordinate so that $z\in\R$ 
and $0<z<\lambda^{-1}$.  
\item
Let $c_2(t):[0,1]\to M$ be a curve starting at $c_2(0)=(1,w_0)$ 
whose projection to the $z$-plane is an embedded curve in $\{\Im(z)<0\}$, 
and whose endpoint $c_2(1)$ has a $z$ coordinate so that $z\in\R$ 
and $z<0$.  
\end{itemize}
With $F(1,w_0)=\id$, we solve Equation \eqref{eq:bryant-dual} along 
these two paths to find 
\[ 
F(c_1(1))=\begin{pmatrix}A_1&B_1\\C_1&D_1\end{pmatrix},
\quad\text{and}\quad
F(c_2(1))=\begin{pmatrix}A_2&B_2\\C_2&D_2\end{pmatrix}.
\]
Let $\tau_j$ be the deck transformation of $\widetilde{M}$ associated to the 
homotopy class of $\gamma_j$ ($j=1,2$). 

\begin{lemma}\label{lem:secondlemma}
We have that $F\circ\tau_1=F\Phi_1$ and $F\circ\tau_2=F\Phi_2$, where 
\begin{align*}
\Phi_1&:=\begin{pmatrix}\bar{A}_1&-\bar{C}_1\\
                       -\bar{B}_1&\bar{D}_1\end{pmatrix}
         \begin{pmatrix}D_1&\Lambda C_1\\\Lambda^{-1}B_1&A_1\end{pmatrix}
         \begin{pmatrix}\bar{D}_1&-\Lambda\bar{B}_1\\
                        -\Lambda^{-1}\bar{C}_1&\bar{A}_1\end{pmatrix}
         \begin{pmatrix}A_1&B_1\\C_1&D_1\end{pmatrix} \; , 
\\ 
\Phi_2&:=\begin{pmatrix}\bar{D}_2&-\bar{B}_2\\-\bar{C}_2&\bar{A}_2\end{pmatrix}
         \begin{pmatrix}A_2&B_2\\C_2&D_2\end{pmatrix} \; .
\end{align*}
\end{lemma}

\begin{proof}
The loop $\gamma_1$ has four portions, as described above.  
The first portion is represented by the curve $c_1$.  The 
second portion is represented by $\kappa_3 \circ c_1^{-1}$.  
Using the facts that the third portion starts at the point 
$(1,\Lambda w_0)$ and that $\kappa_2 \circ \kappa_1 \circ 
\kappa_3(1,w_0) = (1,\Lambda w_0)$, we have that the third 
portion is represented by $\kappa_2 \circ \kappa_1 \circ 
\kappa_3 \circ c_1$.  The final fourth portion is represented 
by $\kappa_1 \circ \kappa_2 \circ c_1^{-1}$, which follows 
from noting that $\kappa_1(1,w_0)=(1,w_0)$, 
$\kappa_2(1,w_0)=(1,w_0)$ and $\kappa_3(1,w_0)=(1,\Lambda w_0)$.  
(In particular, we see that $(1,w_0)$ is in the 
fixed point set of $\kappa_1$ and $\kappa_2$ 
but not in that of $\kappa_3$.)  

Thus we have that 
\[ \gamma_1 = (\kappa_1 \circ \kappa_2 
\circ c_1^{-1}) \circ (\kappa_2 \circ \kappa_1 \circ 
\kappa_3 \circ c_1) \circ (\kappa_3 \circ c_1^{-1}) 
\circ c_1 \; . \]  
Similarly, we can see that 
\[ \gamma_2 = (\kappa_1 \circ c_2^{-1}) \circ c_2 \; . \]  
We can then apply Lemma \ref{lm:rs-first} to get the 
result.  
\end{proof}

Using the Bryant type representation \eqref{eq:bryant-dualB} to make CMC $1$ surfaces in 
$H^3$ and $S_1^3$, the conditions for the resulting surfaces to 
be well defined on $M$ are that $\Phi_1$ and $\Phi_2$ are in 
$\SU(2)$ and $\SU(1,1)$, respectively, and by symmetry, only the 
homotopy classes coming from $\gamma_1$ and $\gamma_2$ need be 
considered.  However, the initial condition $F(1,w_0)=\id$ will 
not cause $\Phi_1$ and $\Phi_2$ to lie in $\SU(2)$ or $\SU(1,1)$.  To 
remedy this, we will change the initial condition for the solution $F$ 
so that it has initial condition
\[ 
F(1,w_0)=\begin{pmatrix} \alpha & \beta \\ \beta & \alpha 
\end{pmatrix} \in \SL(2,\R) 
\] 
in the case the ambient space is $H^3$ (that is, the $\SU(2)$ case), and 
\[ 
F(1,w_0)=\begin{pmatrix} \alpha & \beta \\ \alpha & -\beta
\end{pmatrix} \in \SL(2,\R) 
\] 
in the case the ambient space is $S^3_1$ (that is, the $\SU(1,1)$ case).  

Note that with these changes of initial condition of $F$, we still 
have enough symmetry to conclude that the resulting surfaces are 
well defined on $M$ just by looking only at the two homotopy classes 
represented by $\gamma_1$ and $\gamma_2$.  
This is because the homotopy group of $M$ is generated by 
$[\gamma_1]$ and $[\gamma_2]$.  
The monodromies associated 
to those two homotopy classes are now, for $j=1,2$,  
\[ 
\begin{pmatrix} \alpha & \beta \\ \beta & \alpha \end{pmatrix}^{-1} 
\Phi_j \begin{pmatrix} \alpha & \beta \\ \beta & \alpha \end{pmatrix}
\] 
in the $\SU(2)$ case, and 
\[ 
\begin{pmatrix} \alpha & \beta \\ \alpha & -\beta \end{pmatrix}^{-1} 
\Phi_j \begin{pmatrix} \alpha & \beta \\ \alpha & -\beta \end{pmatrix}
\] 
in the $\SU(1,1)$ case.  

The closing conditions, that is, the conditions that the surfaces are 
well defined on $M$ itself, are now that the above pairs of matrices lie 
in $\SU(2)$ in the first case, and in $\SU(1,1)$ in the second case.  
Noting that $\Phi_1$ and $\Phi_2$ take the forms 
\[ 
\Phi_1 = \begin{pmatrix} r_1 & p \\ -\bar p & r_2 \end{pmatrix} \; , \quad 
\Phi_2 = \begin{pmatrix} q & i r_3 \\ i r_4 & \bar q \end{pmatrix} 
\] 
for complex numbers $p$, $q$ and real numbers 
$r_1$, $r_2$, $r_3$, $r_4$, a direct computation gives 
that the closing conditions are 
\[ 
\frac{2\Re(p)}{r_2-r_1} = \frac{2\Im(q)}{r_4-r_3} 
= \frac{\alpha^2+\beta^2}{2 \alpha \beta} 
= \frac{1+2 \beta^2}{2 \beta \sqrt{1+\beta^2}} 
\in (-\infty,-1) \cup (1,\infty) 
\] 
%(or $r_1=r_2$, $r_3=r_4$, $\beta =0$ and $\alpha =1$)  
in the $\SU(2)$ case, and 
\[ 
\frac{2\Re(p)}{r_2-r_1} = \frac{2\Im(q)}{r_4-r_3} 
= \frac{\alpha^2-\beta^2}{\alpha^2+\beta^2} 
= \frac{1-4 \beta^4}{1+4 \beta^4} 
\in (-1,1) 
\] 
in the $\SU(1,1)$ case.  So we have now proven the following 
lemma: 

\begin{lemma}
The single closing condition for one of the surfaces 
in \eqref{eq:bryant-dualB} is that 
\begin{equation}\label{eq:cond} 
h_1(c,\lambda)=h_2(c,\lambda)\in\R\setminus\{\pm 1\},
\end{equation}
where
\[
h_1(c,\lambda) = \frac{2\Re(p)}{r_2-r_1} \quad\text{and}\quad
h_2(c,\lambda) = \frac{2\Im(q)}{r_4-r_3} 
\] 
holds, and then the appropriate $\alpha$ and $\beta$ can be found.  
Whether one obtains a surface in $H^3$ or 
$S_1^3$ is determined by whether the absolute value of 
$h_1(c,\lambda)=h_2(c,\lambda)$ is greater than or less than $1$.  
\end{lemma}

   \begin{remark}\label{rm:min-max}
    If we consider the minimal surface
    \[
     \Re\int \left( 1-G^2,\,i(1+G^2),\,2G\right)\Omega \in\R^3
    \]
    with the Weierstrass data \eqref{eq:w-data}, one can check that
    the period is solved for the loop $\gamma_1$ (regardless of the 
    choice of $\lambda$, and we chose $G$ as we did in \eqref{eq:w-data} 
    in order to make this true, see Remark 
    \ref{rem:onthechoiceofG}), but is never 
    solved for the loop $\gamma_2$ (for any choice of $\lambda$).
    On the other hand, if we consider the maximal surface
    \[
     \Re\int \left( 1+G^2,\,i(1-G^2),\,2G\right)\Omega \in\R^3_1
    \]
    with the Weierstrass data \eqref{eq:w-data}, 
    the period is solved for $\gamma_2$, but never for 
    $\gamma_1$.  See Figure \ref{catenoidnonclosing}.  
   \end{remark}

\begin{figure} %%%%%%%%%%%%%%%%%%%%%%%%%%%%%%%%%%%%%%%%%%%%%%%%%%%%%%%%%
\begin{center}
\begin{tabular}{cc}
 \includegraphics[width=.40\linewidth]{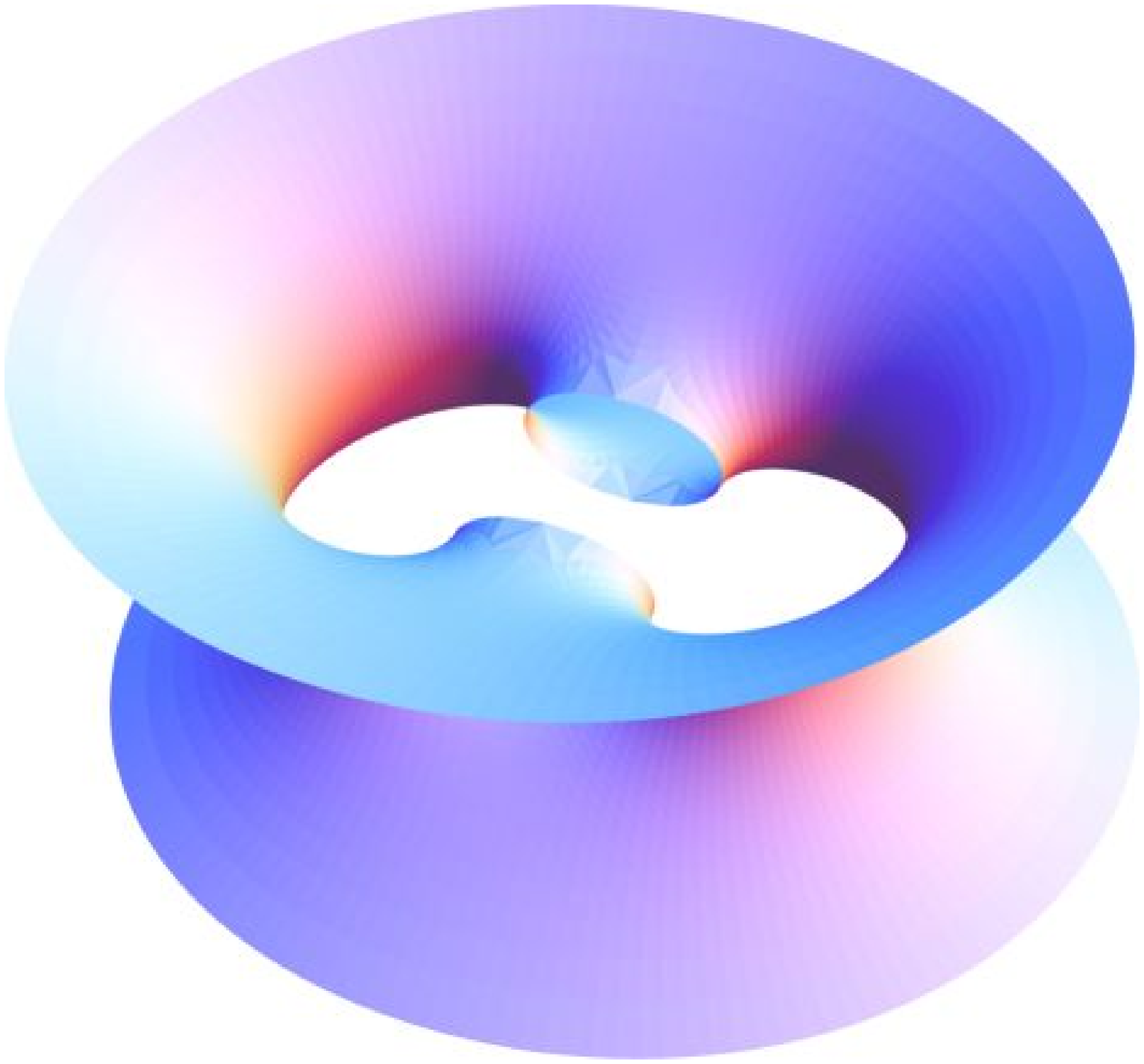} &
 \includegraphics[width=.50\linewidth]{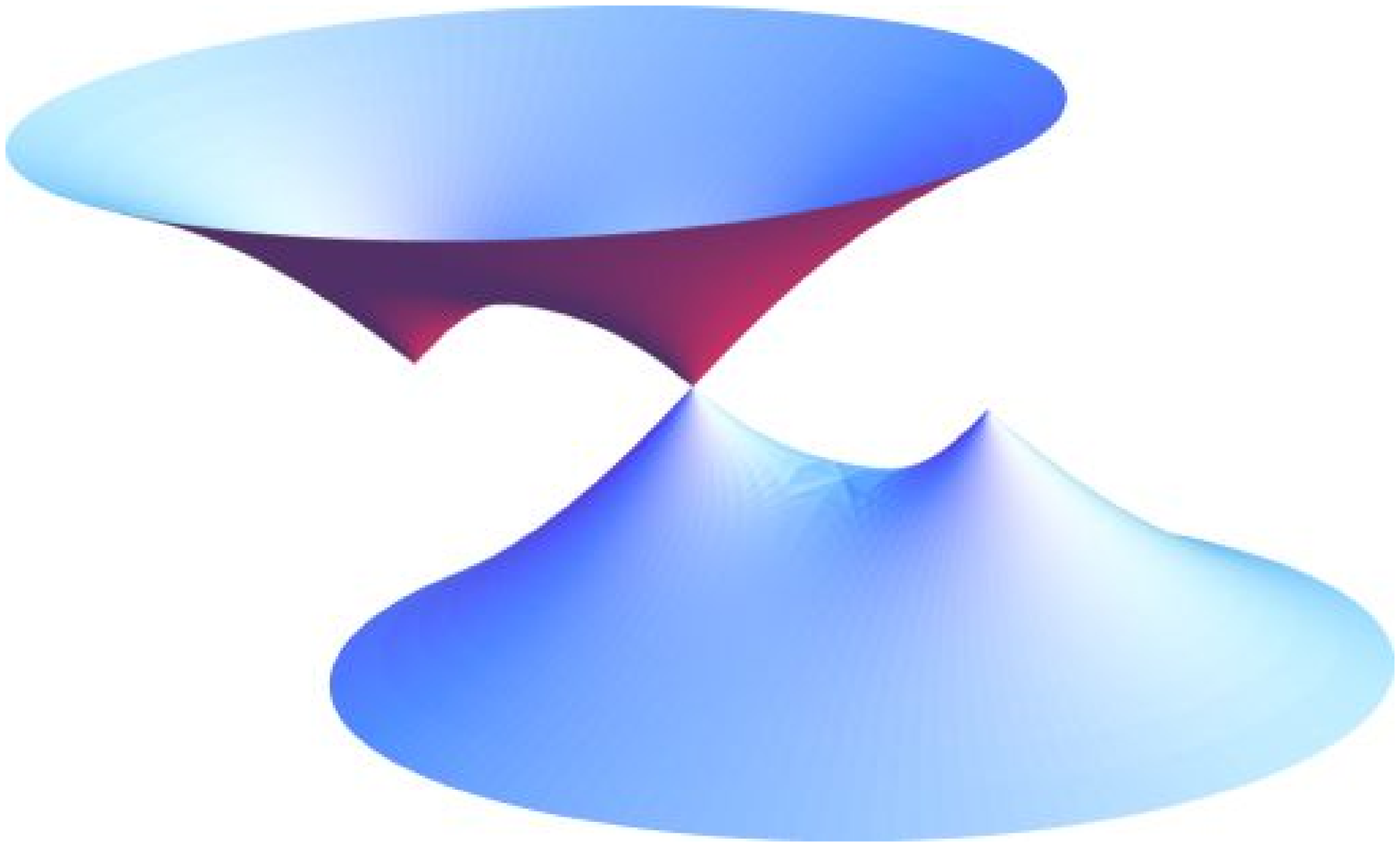} 
\end{tabular}
\end{center}
\caption{A minimal surface in $\R^3$ (left) and a maximal surface 
in $\R_1^3$ (right), constructed as in Remark \ref{rm:min-max}, with 
$k=1$ and $\lambda=2$, in each case showing closing with respect to 
one loop $\gamma_j$ but not with respect to the other.}
\label{catenoidnonclosing}
\end{figure} %%%%%%%%%%%%%%%%%%%%%%%%%%%%%%%%%%%%%%%%%%%%%%%%%%%%%%%%%%%

\section{Numerical experiments and the main result}

Fix $\lambda =2$.  Here we provide constants $c\in\R\setminus\{0\}$ so 
that $h_1(c,2)=h_2(c,2)$ for the genus $k=1,\dots ,20$ cases.  
%(In fact, we have checked existence of such $c$ for all 
%of the cases $k=1,\dots ,20$, but we only show the values of $c$ 
%for $k \leq 10$ here.)  
It is a simple application of the 
intermediate value theorem to show $h_1(c,2)=h_2(c,2)$, 
and the functions $h_j(c,2)$ are stable with respect to numerics 
if the paths $c_i(t)$ are chosen well -- so the numerics are not 
delicate, and are expected to give reliable results.
Furthermore, once we have existence of a surface for one value of 
$\lambda = \lambda_0$, we then have existence for all $\lambda$ 
sufficiently close to $\lambda_0$, so we can conclude existence of 
a $1$-parameter family of such surfaces.

\medskip

\begin{table}
\begin{center}
{\footnotesize 
\begin{tabular}{|c||c|c||c|c||c|c|}
\hline
$k$ & \begin{tabular}{c}$c \in\R\setminus\{0\}$ \\ so that \\ 
                        $h_1(c,2)=\phantom{=}$ \\ $\phantom{==}h_2(c,2)$ \end{tabular} & $h_j(c,2)$ 
    & \begin{tabular}{c}$c \in\R\setminus\{0\}$ \\ so that \\ 
                        $h_1(c,2)=\phantom{=}$ \\ $\phantom{==}h_2(c,2)$ \end{tabular} & $h_j(c,2)$ 
    & \begin{tabular}{c}$c \in\R\setminus\{0\}$ \\ so that \\ 
                        $h_1(c,2)=\phantom{=}$ \\ $\phantom{==}h_2(c,2)$ \end{tabular} & $h_j(c,2)$ 
\\ \hline\hline
1 & $-$0.0467552\phantom{0} & $-$6.91432 
  & $-$0.557726 & \phantom{$-$}0.130869 & %c_2=-0.55772592
       0.704094 & \phantom{$-$}0.221228\phantom{0} \\ \hline
2 & $-$0.0403901\phantom{0} & $-$4.12613 
  & $-$0.505010 & \phantom{$-$}0.218257 & %c_2=-0.50501041
       0.548964 & \phantom{$-$}0.0345248 \\ \hline
3 & $-$0.0334546\phantom{0} & $-$3.32773 
  & $-$0.483326 & \phantom{$-$}0.254392 & %c_2=-0.48332634
       0.482090 & $-$0.0678105 \\ \hline
4 & $-$0.0281931\phantom{0} & $-$2.95960 
  & $-$0.471988 & \phantom{$-$}0.273656 & %c_2=-0.4719884224
       0.444727 & $-$0.132429\phantom{0} \\ \hline
5 & $-$0.0242574\phantom{0} & $-$2.74968 
  & $-$0.465097 & \phantom{$-$}0.285460 & %c_2=-0.465096616
       0.420845 & $-$0.176931\phantom{0} \\ \hline
6 & $-$0.0212467\phantom{0} & $-$2.61454 
  & $-$0.460530 & \phantom{$-$}0.293371 & %c_2=-0.460530215
       0.404255 & $-$0.209443\phantom{0} \\ \hline
7 & $-$0.0188836\phantom{0} & $-$2.52044 
  & $-$0.457291 & \phantom{$-$}0.299018 & %c_2=-0.457291007
       0.392055 & $-$0.234233\phantom{0} \\ \hline
8 & $-$0.0169850\phantom{0} & $-$2.45121 
  & $-$0.454881 & \phantom{$-$}0.303237 & %c_2=-0.45488062
       0.382705 & $-$0.253760\phantom{0} \\ \hline
9 & $-$0.0154287\phantom{0} & $-$2.39818 
  & $-$0.453020 & \phantom{$-$}0.306504 & %c_2=-0.453020375
       0.375309 & $-$0.269538\phantom{0} \\ \hline
10& $-$0.0141310\phantom{0} & $-$2.35627 
  & $-$0.451543 & \phantom{$-$}0.309105 & %c_2=-0.451542993
       0.369312 & $-$0.282553\phantom{0} \\ \hline
11& $-$0.0130330\phantom{0} & $-$2.32232 
  & $-$0.450342 & \phantom{$-$}0.311222 & %c_2=-0.450342302
       0.364352 & $-$0.293472\phantom{0} \\ \hline
12& $-$0.0120924\phantom{0} & $-$2.29427 
  & $-$0.449348 & \phantom{$-$}0.312979 & %c_2=-0.449347816
       0.360180 & $-$0.302764\phantom{0} \\ \hline
13& $-$0.0112778\phantom{0} & $-$2.27070 
  & $-$0.448511 & \phantom{$-$}0.314459 & %c_2=-0.44851098
       0.356623 & $-$0.310766\phantom{0} \\ \hline
14& $-$0.0105655\phantom{0} & $-$2.25062 
  & $-$0.447797 & \phantom{$-$}0.315722 & %c_2=-0.44779729
       0.353553 & $-$0.317730\phantom{0} \\ \hline
15& $-$0.00993749 & $-$2.23331 
  & $-$0.447182 & \phantom{$-$}0.316813 & %c_2=-0.447181578
       0.350878 & $-$0.323846\phantom{0} \\ \hline
16& $-$0.00937975 & $-$2.21824 
  & $-$0.446645 & \phantom{$-$}0.317764 & %c_2=-0.44664507
       0.348525 & $-$0.329260\phantom{0} \\ \hline
17& $-$0.00888113 & $-$2.20499 
  & $-$0.446173 & \phantom{$-$}0.318600 & %c_2=-0.44617348
       0.346439 & $-$0.334086\phantom{0} \\ \hline
18& $-$0.00843272 & $-$2.19326 
  & $-$0.445756 & \phantom{$-$}0.319341 & %c_2=-0.44575574
       0.344578 & $-$0.338415\phantom{0} \\ \hline
19& $-$0.00802733 & $-$2.18280 
  & $-$0.445383 & \phantom{$-$}0.320003 & %c_2=-0.44538317
       0.342907 & $-$0.342319\phantom{0} \\ \hline
20& $-$0.00765905 & $-$2.17341 
  & $-$0.445049 & \phantom{$-$}0.320596 & %c_2=-0.44504884
       0.341398 & $-$0.345859\phantom{0} \\ \hline
\end{tabular}
}
\vspace{\baselineskip}
\caption{Numerical results with $\lambda =2$.  
         The first column gives CMC $1$ surfaces in $H^3$.  
         The second column gives CMC $1$ faces in $S^3_1$ with 
         elliptic ends.  
         The third column gives CMC $1$ faces in $S^3_1$ with 
         hyperbolic ends.}
\label{tb:desitter}
\end{center}
\end{table}

\medskip 

The data in Table \ref{tb:desitter} 
(see also Figure \ref{fg:graphs}) 
then imply the numerical result stated in the introduction, 
except that we still need to analyze the behavior of the ends.  
Because $\text{deg}(G)=k+1$, equality in the Osserman inequality is 
satisfied, for both the $H^3$ and $S_1^3$ cases.  
From this, it follows that the ends are complete and 
embedded in the $H^3$ case, and also in the $S_1^3$ case when the 
ends are elliptic (see \cite{RS} and \cite{F2}, for example).  When 
the ends are hyperbolic in $S_1^3$, they are 
neither complete nor embedded, but are weakly complete, because the 
metric \eqref{liftmetric} is complete, see \cite{FRUYY}.  
However, we have yet to show when the ends are elliptic or 
hyperbolic in the 
$S_1^3$ case.  This final step is taken care of by the next lemma, 
which is similar to arguments found in the appendix of \cite{F2}, 
but here we are allowing for the case of general genus $k$.  

\begin{figure}[htbp] %%%%%%%%%%%%%%%%%%%%%%%%%%%%%%%%%%%%%%%%%%%%%%%%%%%
\begin{center}
\begin{tabular}{ccc}
 \includegraphics[width=.30\linewidth]{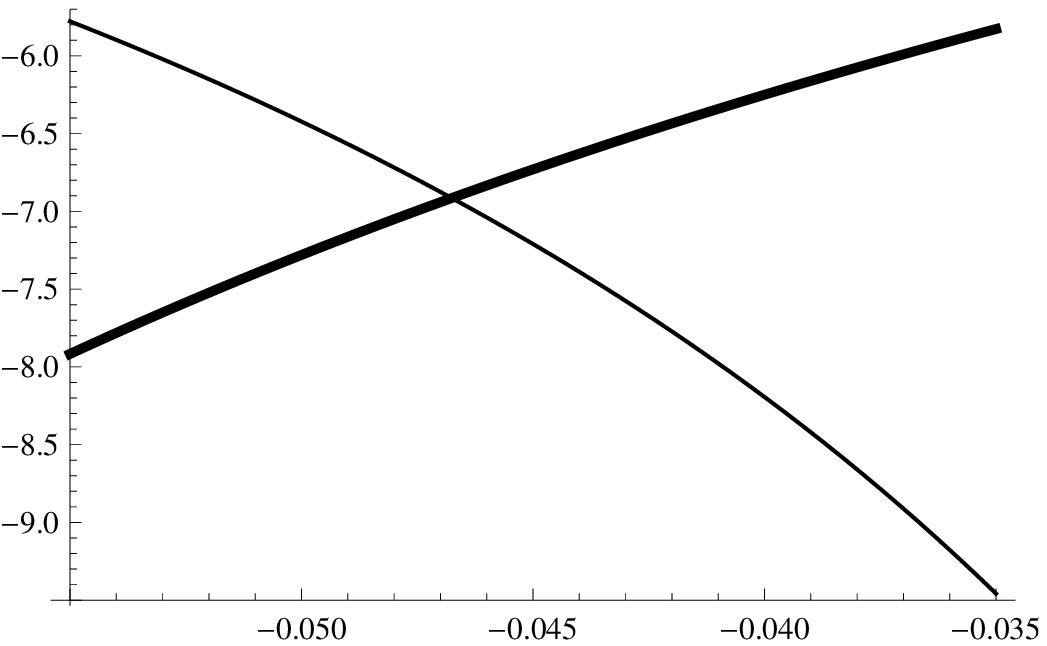} &
 \includegraphics[width=.30\linewidth]{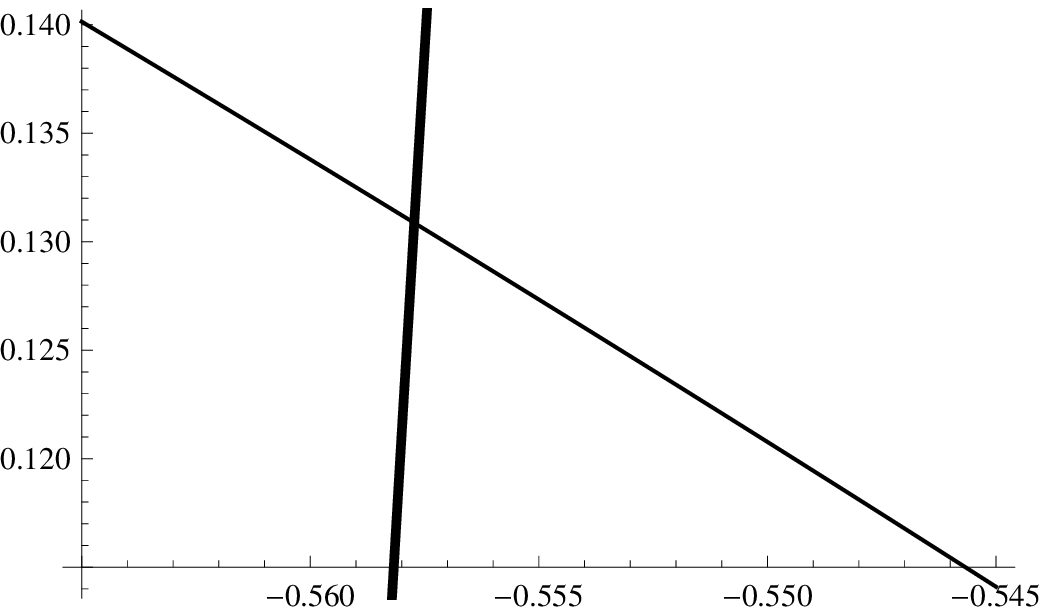} &
 \includegraphics[width=.30\linewidth]{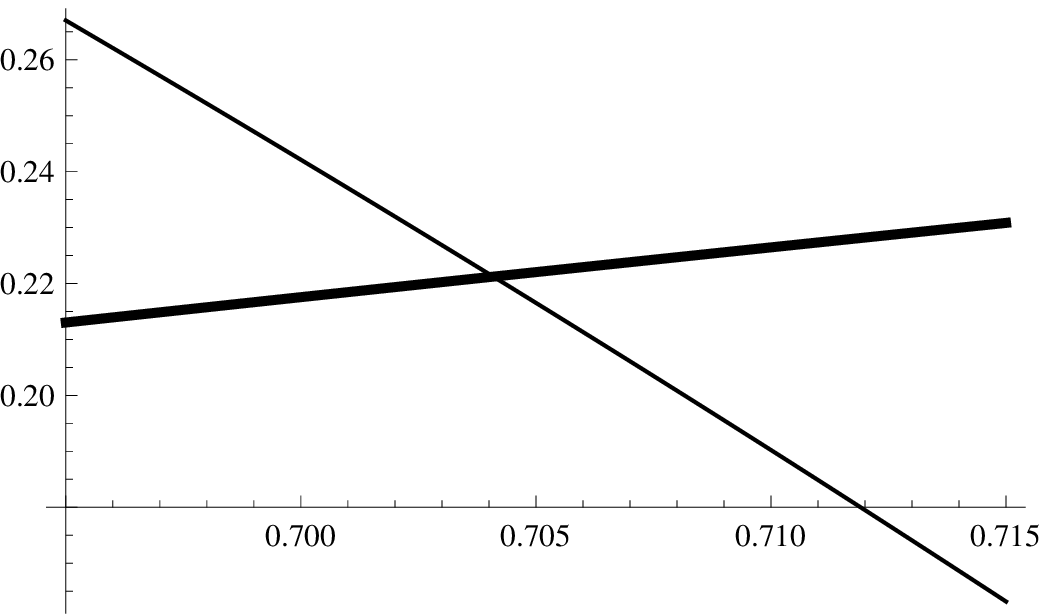} \\
 \includegraphics[width=.30\linewidth]{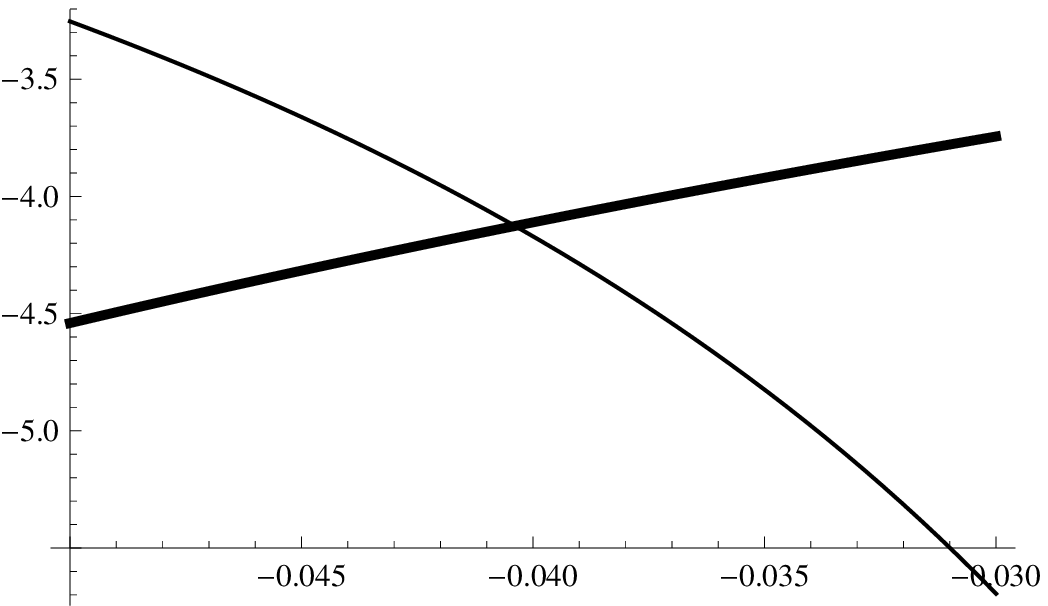} &
 \includegraphics[width=.30\linewidth]{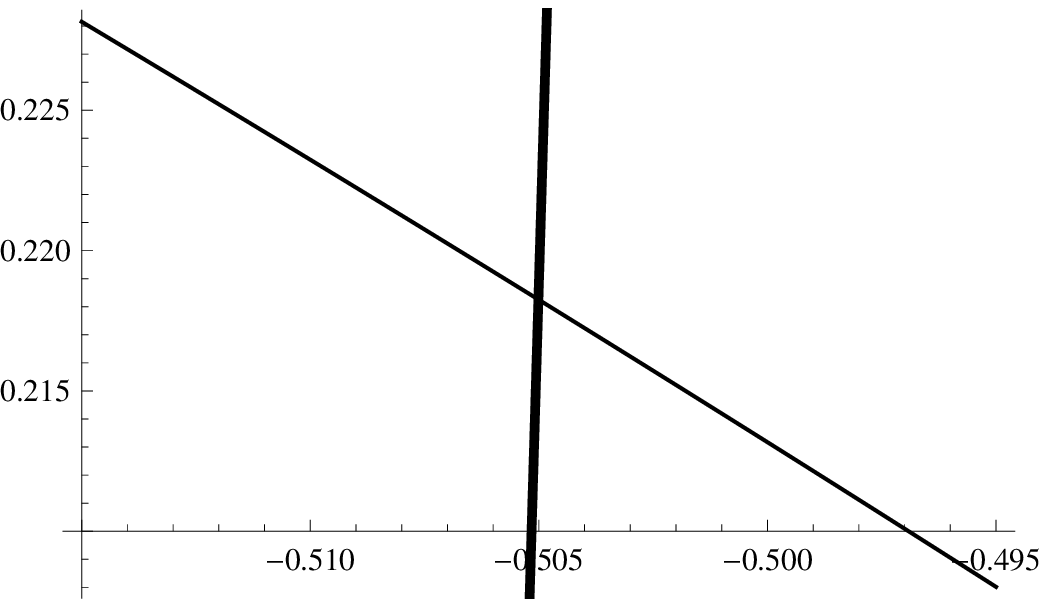} &
 \includegraphics[width=.30\linewidth]{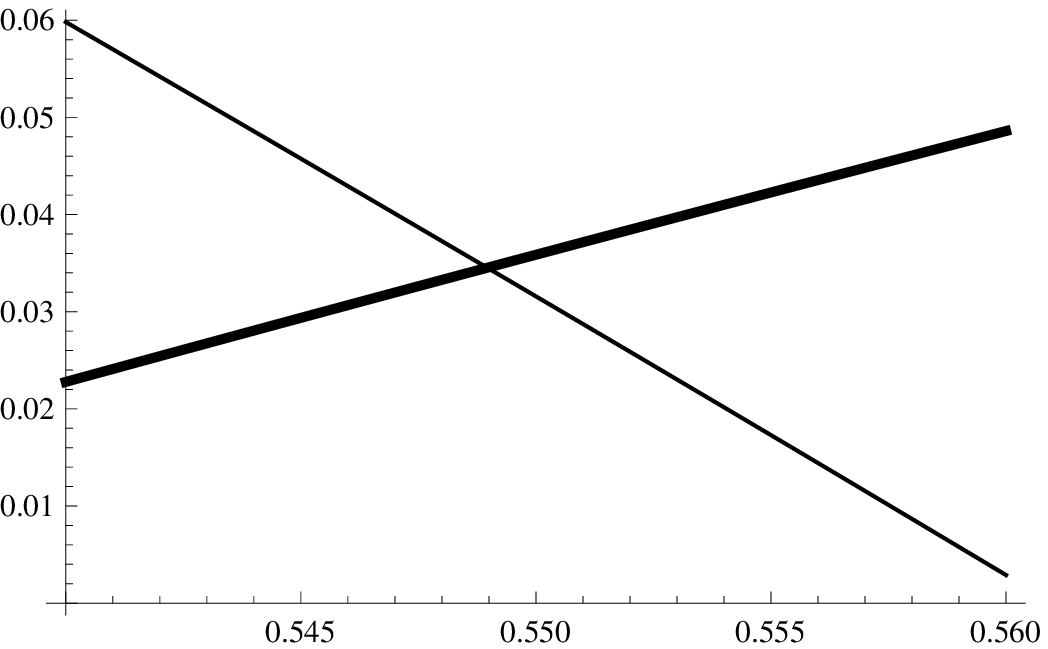} 
\end{tabular}
\end{center}
\caption{Graphs of the functions $h_1(c,2)$ (thin curve) and $h_2(c,2)$ 
(thick curve) in six cases: 
$k=1$ in $H^3$ (upper left), 
$k=1$ in $S_1^3$ with elliptic ends (upper middle), 
$k=1$ in $S_1^3$ with hyperbolic ends (upper right), 
$k=2$ in $H^3$ (lower left), 
$k=2$ in $S_1^3$ with elliptic ends (lower middle), 
$k=2$ in $S_1^3$ with hyperbolic ends (lower right).}
\label{fg:graphs}
\end{figure} %%%%%%%%%%%%%%%%%%%%%%%%%%%%%%%%%%%%%%%%%%%%%%%%%%%%%%%%%%%

\begin{lemma}
The ends of the CMC $1$ 
faces in $S_1^3$ in the middle column $($resp. right hand column$)$ of 
Table~$\ref{tb:desitter}$ have elliptic $($resp. hyperbolic$)$ ends.  
\end{lemma}

\begin{proof}
Let 
\[ F = \begin{pmatrix} A & B \\ C & D 
\end{pmatrix} \] be a solution to \eqref{eq:bryant-dual}.  
Note that, in 
order to determine the type of monodromy about an end, 
we only need to know the eigenvalues of the monodromy 
(provided those eigenvalues are not $\pm 1$), 
and this is independent of the choice of $F$.  So 
we may choose any solution to \eqref{eq:bryant-dual}.  We 
then have 
\[
X_{zz} + \left( \frac{1}{z} - \frac{w_z}{w} \right) X_z + 
\lambda^{k/(k+1)} \frac{c w_z}{z w} X = 0 \; , \;\;\; X = A,B \; , 
\]
and 
\[
Y_{zz} + \left( \frac{1}{z} + \frac{w_z}{w} \right) Y_z + 
\lambda^{k/(k+1)} \frac{c w_z}{z w} Y = 0 \; , \;\;\; Y = C,D \; . 
\]
Because of the symmetry $\kappa_2(z,w)$, 
it suffices to determine the type of 
just one end, and then the other end will automatically have the 
same type.  
So let us choose the end $(z,w)=(0,0)$.  At this end, 
$w$ is a local coordinate for the Riemann surface $\overline{M}$.  
In terms of $w$, and considering $z$ as a function of $w$, the 
equations above become 
\[
X_{ww} + o(1) \frac{X_w}{w} 
       + \lambda^{k/(k+1)} c (k+1) \left( 1 + o(1) \right) \frac{X}{w^2} = 0 
\; , \;\;\; X = A,B \; , 
\]
and 
\[
Y_{ww} + 2 \left( 1 + o(1) \right) \frac{Y_w}{w} 
       + \lambda^{k/(k+1)} c (k+1) \left( 1 + o(1) \right) \frac{Y}{w^2} = 0 
\; , \;\;\; Y = C,D \; , 
\]
where $o(1)$ denotes the Landau symbol, that is, $o(1)$ is a holomorphic 
function $\phi(w)$ around $(z,w)=(0,0)$ so that $\phi(0)=0$.  
It follows that the difference of the solutions of the 
indicial equation corresponding to the first of these two 
equations is 
\[ 
 \sqrt{1 - 4 c (k+1) \lambda^{k/(k+1)}} \; . 
\]
Likewise, the difference of solutions of the indicial equation 
for the second equation above takes the 
same value.  It follows (see the appendix of \cite{F2} for 
further details) that the end is elliptic (resp. hyperbolic) if 
\[ 
 1 - 4 c (k+1) \lambda^{k/(k+1)} 
\] 
is positive (resp. negative) which is indeed the case for the data given 
in the middle column (resp. right hand column) of Table~\ref{tb:desitter}.  
\end{proof}

%%%%%%%%%%%%%%%%%%%%%%%%%%%%%%%%%%%%%%%%%%%%%%%%%%%%%%%%%%%%%%

\end{document}